\documentclass[article,oneside,oldfontcommands,b5paper,10pt]{memoir}

\counterwithout{section}{chapter}
\setsecnumdepth{subsubsection}

\setsecheadstyle{\bf\raggedright}
\setsubsecheadstyle{\sl\raggedright}
\setsubsubsecheadstyle{\sl\raggedright}
\setsecnumformat{\csname the#1\endcsname.\;}

\usepackage[affil-it]{authblk}

\usepackage{indentfirst}
\usepackage[margin=1.05 in]{geometry}

\usepackage[numbers,sort&compress]{natbib}
\usepackage[lined,boxed,ruled,commentsnumbered]{algorithm2e} 
\usepackage{algorithmicx}

\usepackage{amssymb,amsmath,mathrsfs}
\usepackage{amsthm,stmaryrd}
\usepackage[hypertexnames=false,colorlinks,linkcolor={blue},citecolor={blue}]{hyperref}
\usepackage{graphicx,float,caption}

\usepackage{epsfig,subcaption}

\usepackage{scalerel} 

\usepackage{enumitem}
\setlist{
  leftmargin=*, 
  topsep = 4pt, 
  itemsep = 1pt, 
}

\usepackage[T1]{fontenc}

\usepackage{fancyhdr}
\fancypagestyle{basic}{
  \fancyhf{}
  \fancyhead[L]{\thetitle}
  \fancyhead[R]{\thepage}
}
\fancypagestyle{draftfirstpage}{
  
  \fancyhf{}
}

\numberwithin{equation}{section}

\newtheorem{theorem}{Theorem}[section]

\newtheorem{lemma}[theorem]{Lemma}

\theoremstyle{definition}
\newtheorem{definition}[theorem]{Definition}

\newtheorem{assumption}{Assumption}

\theoremstyle{remark}

\newtheorem{example}[theorem]{Example}

\allowdisplaybreaks

\renewcommand{\qed}{\hfill\rule{1ex}{1.5ex}}

\newcommand{\reDeclareMathOperator}[2]{\let#1\undefined \DeclareMathOperator{#1}{#2}}
\newcommand{\reDeclareMathOperatorL}[2]{\let#1\undefined \DeclareMathOperator*{#1}{#2}}

\reDeclareMathOperator{\mod}{mod}

\reDeclareMathOperator{\supp}{supp}

\reDeclareMathOperator{\Proj}{Proj}

\renewcommand{\bar}{\overline}

\let\originalbigcap\bigcap
\let\originalbigcup\bigcup
\reDeclareMathOperatorL{\bigcap}{\mathbin{\scaleobj{0.9}{\originalbigcap}}}
\reDeclareMathOperatorL{\bigcup}{\mathbin{\scaleobj{0.9}{\originalbigcup}}}

\makeatletter
\renewcommand*{\@fnsymbol}[1]{\ifcase#1\or*\else\@arabic{#1}\fi}
\makeatother

\title{A self-adaptive subgradient extragradient method with conjugate gradient-type direction for pseudomonotone variational inequalities}

\author[1,3]{Ibrahim Arzuka}
\author[1,2]{Parin Chaipunya\thanks{Corresponding author}$^{,}$}
\author[1,2]{Poom Kumam}

\affil[1]{
  {Department of Mathematics, Faculty of Science, King Mongkut's University of Technology Thonburi},
  {126 Pracha Uthit Rd.},
  {Bang Mod, Thung Khru, Bangkok},
  {10140}, 
  {Thailand.}
}
\affil[2]{
  {Center of Excellence in Theoretical and Computational Science (TaCS-CoE), King Mongkut's University of Technology Thonburi},
  {126 Pracha Uthit Rd.},
  {Bang Mod, Thung Khru, Bangkok},
  {10140}, 
  {Thailand.}
}
\affil[3]{
{Department of Mathematics, Faculty of  Science, Sa'adu Zungur University, Bauchi, Nigeria.}}
\affil[ ]{
  Emails: arzuka2000@gmail.com, parin.cha@kmutt.ac.th and  poom.kum@kmutt.ac.th;
  \vspace{.25cm}
}

\date{}

\begin{document}
\maketitle \vspace{-1.8cm}
\thispagestyle{draftfirstpage}

\newcommand{\pre}{\preccurlyeq_{\varepsilon,\varphi}}

\begin{abstract}
  This paper introduces a subgradient extragradient algorithm with a conjugate gradient‐type direction to solve pseudomonotone variational inequality problems in Hilbert spaces. The algorithm features a self-adaptive strategy that eliminates the need for prior knowledge of the Lipschitz constant and incorporate a conjugate gradient‐type direction to enhance convergence speed. We establish a result describing the behavior of sequences generated therefrom toward the solution set. Using this result, we prove the strong convergence of the proposed method and provide numerical experiments to demonstrate its computational efficacy and robustness. Finally, we present a potential application of the method.



  \noindent{\bf Keywords:} Conjugate gradient; Pseudomonotone operators; Variational inequality.

\end{abstract}


\normalsize

\section{Introduction}\label{sec1}
The variational inequality problem (brieftly, VIP) is concerned with finding a point
$x^{*}\in C$ such that   
\begin{equation}\label{sg1}
  \left\langle B(x^{*}), x-x^{*}\right\rangle\geq 0, \quad \forall x\in C, 
\end{equation}
where $B$ is an operator defined on a real Hilbert space $H$ into itself, and $C$ is a closed, convex, nonempty subset of $H$. The VIP  is regarded as a unifying tool that generalizes the study of several problems into a common form, including those in engineering, transportation, economics, mathematical programming, and mechanics, as stated in
\cite{solodov1999,yom2021inertial,iiduka2012fixeda,iiduka2012afixed,2012medical,adamu2022accel,saberi2016network,semenov2005vector,konnov1997systems,kravchuk2007variational}. Due to the wide range applications of the VIP, numerous methods have been proposed in the literature to determine its solution. The extragradient method (EGM) \cite{korpelevich19} is a well-known technique for this purpose:  
\begin{equation}\label{pp0}
  \begin{cases}
    &x_0 \in C,\\
    &y_j=P_{C}(x_j-\lambda B{(x_j)}),\\
    &x_{j+1}=P_{C}(x_j-\lambda B{(y_j)}),
  \end{cases}
\end{equation}
where \( \lambda \in \left(0, \frac{1}{L}\right) \) and $B$ is a monotone oprator. The sequence \( \{x_j\} \) generated by EGM converges weakly to the solution of the VIP. But the operator \( B \) and the metric projection \( P_{C} \) are computed twice at each iteration, and the step length \( \{\lambda_{j}\} \) heavily depends on the Lipschitz constant \( L \) of the underlying operator. This leads to the method being computationally expensive, particularly when \( P_{C} \) is not explicitly known or when the Lipschitz constant is unknown in advance. 

To address this limitation, Popov \cite{popov1980mod} proposed the following modification:
\begin{equation}
  \begin{cases}
    &x_0 \in C,\\
    &y_j=P_{C}(x_j-\lambda_j B{(x_j)}),\\
    &x_{j+1}=P_{C}(y_j-\lambda_j B{(x_j)}),
  \end{cases}
\end{equation}
where \( \lambda_j \in \left(0, \frac{1}{3L}\right) \). The sequence \( \{x_{j}\} \) of approximations generated by this method converges weakly to the solution of the VIP. This strategy reduces the computation of the operator \( B \) to only once per iteration, but it still requires computing the metric projection \( P_C \) twice at each iteration. Censor et.al \cite{censor2011strong,censor2012} proposed the subgradient extragradient method (SEGM) by replacing the second projection in \eqref{pp0} with a projection onto the half space as:
\begin{equation}
  \begin{cases}
&x_0 \in C,\\
&y_j=P_{C}(x_j-\lambda B{(x_j)}),\\
&V_{x}=\{\mu\in H:  \left\langle x_j-\lambda_jB(t)-y_j, \mu-y_{j}\right\rangle\\
&x_{j+1}=P_{V_{x}}(y_j-\lambda B{(x_j)}),
  \end{cases}
\end{equation}
where $\lambda_j\in (0,\frac{1}{L}).$ This strategy significantly reduces the extragradient method's computational burden, as the half-space's closed-form expression is always known. It converges weakly when \( B \) exhibits monotonicity, and Lipschitz continuity assumptions. 

Over the years, several improved versions of the SEGM have been proposed to solve variational inequalities, equilibrium problems, and other optimization problems. In this context, inertial approaches have been introduced to accelerate the convergence rate of related SEGM algorithms.
On the other hand, some algorithms are designed such that their step size is determined without prior knowledge of the Lipschitz constants, and no any line search procedure is needed. This approach, known as self-adaptive, provides additional computational advantages compared to other methods in the literature.  
For instance, 
Yang et al. \cite{yang2018modified} proposed a self-adaptive version of SEGM  as: 
\[
  \begin{aligned}
    \begin{cases}
      y_j = P_C(x_j - \alpha_j B (x_j)), \\
      T_j = \left\{ w \in H \mid \langle x_j - \alpha_j B( x_j) - y_j, w - y_j \rangle \leq 0 \right\}, \\
      x_{j+1} = P_{T_j}(x_j - \alpha_j B (y_j), \\
      \alpha_{j+1} = \begin{cases}
        \min \left( \frac{\mu \left( \|x_j - y_j\|^2 + \|x_{j+1} - y_j\|^2 \right)}{2 \langle B( x_j) - B (y_j), x_{j+1} - y_j \rangle}, \lambda_j \right), & \text{if } \langle B(x_j) - B (y_j), x_{j+1} - y_j \rangle > 0, \\
        \lambda_j, & \text{otherwise,}
      \end{cases}
    \end{cases}
  \end{aligned}
\]
and showed that the sequence $\{x_j\}$ generated   converges weakly to a solution of $VIP(B, C).$  For more into subgradient extragrediedent methods, refer to the manuscripts and the reference therein \cite{abubakar2022self,wang2022self,ur2022inertial,alakoya2021modified,tian2019self}. 
Similarly, the conjugate gradient-type methods were proposed to solve variational inequality problems. These methods converge strongly to the solution of the VIP under stringent assumptions that the operator is $\alpha$ strongly monotone and the feasible set is bounded see for example the manuscripts \cite{iiduka2015acce,iiduka2009use,iiduka2011three}.
Motivated by the aforementioned studies, we relax the stringent assumptions to propose a subgradient extragradient method with a conjugate gradient-type direction for solving a pseudomonotone variational inequality problem in Hilbert space. Moreover, the scheme is equipped with a self-adaptive strategy that eliminates the need for prior knowledge of the Lipschitz constant. To our knowledge, this is the first approach to incorporate a conjugate gradient-type direction into the subgradient extragradient method.

The structure of this paper is as follows. Section \ref{zec} presents the essential definitions and lemmas required for the subsequent discussions. In Section \ref{zec2}, we introduce the algorithm and analyze its convergence. Section \ref{zec3} provides numerical examples with application to the noncooperative matrix game problem, demonstrating the efficacy of the proposed algorithm compared to state-of-the-art methods.
The paper concludes with a concise summary in Section \ref{zec4}.

\section{Preliminaries}\label{zec}
In this section, we focus on the definitions and lemmas necessary for this study. Unless otherwise stated, we consistently represent the solution set of the variational inequality problem with the operator \( B \) over the set \( C \) as \( VIP(B, C) \) throughout the manuscript.

\begin{definition}
  Let $B : H \rightarrow H$  be a mapping then
  \begin{enumerate}[label=\upshape(\roman*)]
    \item $B$ is  $\alpha$-strong  monotone operator if there exist $\alpha>0$ such that $$\langle x-y \; , \; B(x)-B(y) \rangle \geq  \alpha \|x- y \|^{2},\quad \forall x,y \in H.$$ 
    \item $B$ is  monotone if  for any $x,y \in H$  $$\langle x-y \; , \; B(x)-B(y) \rangle \geq  0.$$  
    \item $B$ is pseudomonotone if  for any $x,y \in H$  $$\langle x-y \; , \; B(x) \rangle \geq  0\;\; \Rightarrow\; \;\langle x-y \; , \; B(y) \rangle \geq  0.$$ 
    \item $B$ is  lipschitz continuous if  there exist $L>0$ such that  $$\|B(x)- B(y) \|\leq L \|x- y \|, \quad \forall x,y\in C.$$ 
  \end{enumerate}
\end{definition}  

\begin{definition}  \label{def01}
  Let us consider a nonempty, closed, convex subset \( C \) of the Hilbert space \( H \). The mapping \( P_{C}: C \rightarrow H \), which assigns each element of \( H \) to its unique nearest element in \( C \), is referred to as a metric projection. Some properties of this mapping include:
  \begin{enumerate}[label=\upshape(\roman*),leftmargin=*]
    \item  $\langle x-P_{C}(x) \; , \; y-P_{C}(x) \rangle \leq 0,~~\forall~ y \in C ~ and ~ x\in H,$
    \item  $\|x- P_{C}(y) \|= \inf\|x- y \|, ~~~ \forall y \in H, $
    \item $\|P_{C}(x)- P_{C}(y) \|\leq\|x- y \|, ~~~ \forall x,y \in H.$
    \item $\|P_{C}(x)-y\|^{2}\leq \|x-y\|^{2}-\|P_{C}(x)-x\|^{2}, ~~~ \forall y \in H.$
  \end{enumerate}
\end{definition}  
The following characterizations would be used to establish the proposed method's convergence.
\begin{lemma}\label{l1}
  Let $x,y\in H$ and $\alpha \in \mathbb{R}$.   Then
  \begin{enumerate}[label=\upshape(\roman*)]
    \item  $\|x+y\|^2=\|x\|^2+\|y\|^2+2\langle x \; , \;y \rangle$.
    \item $\|x+y\|^2\leq\|x\|^2+2\langle x+y \; , \;y \rangle$.
    \item $\|(1-\alpha)x-\alpha y\|^2 =(1-\alpha)\|x\|^2+\alpha\|y\|^2-\alpha(1-\alpha)\|x-y\|^2$.
  \end{enumerate}
\end{lemma}
\begin{lemma}[{\cite[Lemma 8]{he2013solving}}]\label{con}
  Suppose that $\{S_j\}$ is a sequence of nonnegative real numbers such that
  \[ S_{j+1} \leq (1 - \alpha_j) S_j + \alpha_j V_j, \quad \forall n \geq 0, \]
  and 
  \[ S_{j+1} \leq S_j - \xi_j + \Theta_j, \quad \forall n \geq 0, \]
  where $\{\alpha_j\}$ is a sequence in $(0, 1)$, $\{\xi_j\}$ is a sequence of nonnegative real numbers, $\{V_j\}$ and $\{\Theta_j\}$ are real sequences such that:
  \begin{enumerate}[label=\upshape(\roman*),leftmargin=*]
    \item $\sum_{j=0}^{\infty} \alpha_j = \infty$,
    \item $\underset{j \to \infty}{\lim}\Theta_j = 0$,
    \item $\underset{k \to \infty}{\lim }\xi_{j_k} = 0$ implies $\underset{k \to \infty}{\limsup} V_{j_k} \leq 0$ for any subsequence $\{\xi_{j_k}\}$ of $\{\xi_j\}$.
  \end{enumerate}
  Then $\underset{k \to \infty}{\lim }S_{j} = 0$.
\end{lemma}

\begin{lemma}[{\cite[Lemma 1]{opial1967weak}}] \label{opial} 
  If a sequence $(x_j)_{j\in\mathbb{N}}\subset H$  convergence weakly to $x$, then for any $y\neq x\in H$ 
  $$\liminf_{n\rightarrow{\infty}}\|x_j-y\|> \liminf_{n\rightarrow{\infty}}\|x_j-x\|.$$
\end{lemma}

\section{The  Algorithm }\label{zec2}
To construct and establish the convergence of the  proposed method, we require the conditions of the following assumption.
\begin{assumption}
  We make the following assumptions
  \begin{enumerate}[label=\upshape($A_{\arabic*}$)]
    \item  $C\subset H$ is a nonempty, closed, and convex set. 
    \item  The solution set $VIP(B,C)$ is nonempty.
    \item  $B:{H}\rightarrow{H}$ is pseudomonotone and Lipschitz continuous.
    \item  $\mu\in(0,1),$\, $\alpha_{j}\in (0,1)$\, such that\,
      $\sum^{\infty}_{j = 0} \alpha_j = \infty$ and  $\underset{j \to \infty}{\lim} \alpha_j = 0.$\,  $\Gamma_{j}\in [0,+\infty)$ \,such that\,
      $\sum^{\infty}_{j = 0} \Gamma_{j} < \infty$.
  \end{enumerate}
\end{assumption}

\begin{algorithm}[H]
  \caption{A self-adaptive subgradient extragradient method with conjugate gradient-type direction }
  \begin{algorithmic}
    \State \textbf{Initialization:} Select  $\mu,\,\psi, \, \{\Gamma_j\}\, \text{and}\,  \{\alpha_j\}$ such that condition $(A_{4})$ holds.\, Choose $x_{0}\in H,\, \rho \in C$ and $d_{0}=-B(x_0).$  Set $j=0$
    \State \quad {\bf Step 1:} Determine 
    $$\Theta_j=\dfrac{\Gamma_{j}}{\max\{\|d_{j}\|,\psi\}}$$ and
    \begin{equation}\label{T0} 
      d_{j+1}= -B(x_{j})+\Theta_{j} d_{j}.
    \end{equation}
    \State \quad {\bf Step 2:} Compute 
    $$w_j=P_{C}(x_j+\lambda_j d_{j+1})$$ and
    $$y_j=P_{T_{j}}(x_j-\lambda_j B(w_j)).$$
    $$ T_j := \{ z \in H : \langle x_j-\lambda_j B(w_j) - y_j, z - y_j \rangle \geq 0. \}$$
    \State \quad {\bf Step 3:} Compute  $x_{j+1}\in H$ as 
    \begin{equation}\label{M22}
      x_{j+1}=\alpha_j \rho +(1-\alpha_j)y_j,
    \end{equation}
    and update the stepsize $\lambda_{j+1}$ by 
    \[
      \lambda_{j+1}  =
      \begin{cases}\label{lam}
        \min \left\{\frac{\mu \|w_{j}-x_{j}\|}{\|B(w_{j})-B(x_{j})\|},\lambda_{j}\right\}, & \text{if } B(w_{j})-B(x_{j})\neq 0, \\
        \lambda_{j}, & \text{otherwise}.
      \end{cases}
      \tag{23}
    \]

    \quad   Set $j=j+1$ and go back to step 1.
  \end{algorithmic}
  \label{algo_W}
\end{algorithm}

Now, before we proceed let us  show that the sequence generated by \eqref{lam} is nonincreasing
\begin{lemma}
  Let $\{\lambda_{j}\}$ be sequence of step-lenghts 
  generated by \eqref{lam}. Then, $\{\lambda_{j}\}$ is nonincreasing and $\lambda_{j}\geq \frac{\mu}{L}.$
\end{lemma}
\begin{proof}
  By the lipschitz contiunity of $B,$ we have  $\|B(x_{j})-B(y_{j})\|\leq L\|x_{j}-y_{j}\|$ and together with \eqref{lam}, we get
  $$\frac{\mu\|x_{j}-y_{j}\|}{\|B(x_{j})-B(y_{j})\|}\geq \frac{\mu\|x_{j}-y_{j}\|}{L\|x_{j}-y_{j}\|}=\frac{\mu}{L}.$$
  We find from \eqref{lam} that $\lambda_{j+1}\geq \min\big\{\frac{\mu}{L},\lambda_{j}\big\}.$ By induction, we get that $\lambda_{j}\geq \min\big\{\frac{\mu}{L},\lambda_{1}\big\}.$
  We also see from  \eqref{lam} that $\lambda_{j+1}\leq \lambda_{j} \; \forall j\in \mathbb{N}.$ Combining the monotonicity and the existence of the lower bound of $\{\lambda_j\},$ we find that the limit of $\{\lambda_j\}$ exists. Since $\lambda_{j}\geq \min\big\{\frac{\mu}{L},\lambda_{1}\big\},$ $\lambda>0$ exists such that
  $\underset{j \to \infty}{\lim}\lambda_{j}=\lambda.$
\end{proof}

Now, let begin the  convergence of the  Algorithm \ref{algo_W} with the following lemma.
\begin{lemma} \label{L1}
  Let $\{w_{j}\},\; \{y_{j}\}\; \text{and}\; \{x_{j}\}$ be the sequences generated by Algorithm \ref{algo_W}. Then, for any $z\in VIP(B,C)$, the sequence $\{||x_{j}-z||\}$ is bounded.
\end{lemma}
\begin{proof}
  Let $z\in VIP(B,C),$ $B$ is pseudomonotone implies 
  \begin{equation}
    \langle B(w_j) \; , \;w_j-z \rangle\geq 0.   
  \end{equation}
  which implies 
  \begin{equation}\label{T}
    \langle B(w_j) \; , \;z-w_j \rangle\leq 0.   
  \end{equation}
  Observe that
  \begin{equation}\label{T4}
    \begin{split}
      \langle B(w_j) \; , \;z-y_{j} \rangle&=\langle B(w_j) \; , \;z-w_j \rangle+\langle B(w_j) \; , \;w_j-y_{j} \rangle\\ 
                                           &\leq\langle B(w_j) \; , \;w_j-y_{j} \rangle.
    \end{split}
  \end{equation}
  By \eqref{T4}, we have 
  \begin{equation}\label{T6}
    \begin{split}
      \|y_{j}-z\|^2
        &\leq\|x_j-\lambda_j B(w_j)-z\|^{2} -\|x_j-\lambda_j B(w_j)-y_{j}\|^{2}\\ 
        &=\|x_j-z\|^{2}+\lambda_j^{2}\|B(w_j)\|^{2}-2\lambda_j\langle (x_j-z) \; , \;B(w_j)\rangle\\
        &\quad-(\|x_j-y_{j}\|^{2}+\lambda_j^{2}\|B(w_j)\|^{2}-2\lambda_j\langle (x_j-y_{j}) \; , \;B(w_j)\rangle) \\
        &=\|x_j-z\|^{2}+\lambda_j^{2}\|B(w_j)\|^{2}-2\lambda_j\langle (x_j-z) \; , \;B(w_j)\rangle\\
        &\quad-\|x_j-y_{j}\|^{2}-\lambda_j^{2}\|B(w_j)\|^{2}+2\lambda_j\langle (x_j-y_{j}) \; , \;B(w_j)\rangle) \\
        &=\|x_j-z\|^{2}+2\lambda_j\langle (z-x_j) \; , \;B(w_j)\rangle-\|x_j-y_{j}\|^{2}\\
        &\quad+2\lambda_j\langle (x_j-y_{j}) \; , \;B(w_j)\rangle) \\
        &=\|x_j-z\|^{2}-\|x_j-y_{j}\|^{2}+2\lambda_j\langle (z-y_{j}) \; , \;B(w_j)\rangle)\\
        &\leq\|x_j-z\|^{2}-\|x_j-w_j\|^{2}-\|w_j-y_{j}\|^{2}\\
        &\quad-2\langle x_j-w_j \; , \;w_j-y_{j} \rangle+2\lambda_j\langle B(w_j) \; , \;w_j-y_{j} \rangle\\
        &\leq\|x_j-z\|^{2}-\|x_j-w_j\|^{2}-\|w_j-y_{j}\|^{2}\\
        &\quad+2\langle x_j-w_j \; , \;y_{j}-w_j \rangle-2\lambda_j\langle B(w_j) \; , \;y_{j}-w_j \rangle\\
        &\leq\|x_j-z\|^{2}-\|x_j-w_j\|^{2}-\|w_j-y_{j}\|^{2}\\
        &\quad+2\langle x_j-\lambda_j B(w_j)-w_j \; , \;y_{j}-w_j \rangle.
\end{split}
\end{equation}
From \eqref{T6}, we have 
\begin{equation}\label{TT01}
  \begin{split}
    & 2\langle x_j-\lambda_j B(w_j)-w_j \; , \;y_{j}-w_j \rangle  \\
    &\quad= 2\langle x_j-\lambda_j B(x_j)+\lambda_j\Theta_j d_{j}-w_j \; , \;y_{j}-w_j \rangle\\&\qquad+2\lambda_j\langle  B(x_j)-B(w_j) \; , \;y_{j}-w_j \rangle-2\lambda_j\Theta_j\langle  d_{j}  \; , \;y_{j}-w_j \rangle\\
    &\quad\leq 2\lambda_j\|  B(x_j)-B(w_j) \|\|y_{j}-w_j \|+2\lambda_j\Theta_j\| d_{j}  \| \|y_{j}-w_j \|\\
    &\quad\leq \frac{2\mu\lambda_j}{\lambda_{j+1}}\|  x_j-w_j \|\|y_{j}-w_j \|+2\lambda_j\Gamma_{j} \|y_{j}-w_j \|\\
    &\quad\leq \frac{\mu\lambda_j}{\lambda_{j+1}}\big(\|  x_j-w_j \|^{2}+\|y_{j}-w_j \|^{2}\big)+\lambda_j\Gamma_{j} \|y_{j}-w_j \|^{2}+\lambda_j\Gamma_{j}.
  \end{split}
\end{equation}
Combining \eqref{T6} and \eqref{TT01}, we  have
\begin{equation}\label{TT02}
  \begin{split}
    \|y_{j}-z\|^2&\leq\|x_j-z\|^{2}-q_{j}\big(\|x_j-w_j\|^{2}+\|w_j-y_{j}\|^{2}\big)+\lambda_j\Gamma_{j},
  \end{split}
\end{equation}
where $q_{j}=\Big(1-\Big(\dfrac{\mu\lambda_j}{\lambda_{j+1}}+\lambda_j\Gamma_{j}\Big)\Big)\,\, \forall j\geq 1.\;$ Since $\underset{j \to \infty}{\lim}\lambda_{j}=\lambda >0$ and $\underset{j \to \infty}{\lim}\Gamma_{j}=0,$
we obtain,
$$\underset{j \to \infty}{\lim}q_{j}=(1-\mu). $$
Since   $\mu \in (0,1),$
there exists a number $\bar{N}\in \mathbb{N}$ such that $q_{j}>0$ for all $j\geq \bar{N}.$
Thus, we obtain from Algorithm \ref{algo_W}  that,
\begin{equation}\label{T7}
  \begin{split}
    \|x_{j+1}-z\|^{2}&
    \leq \alpha_j \|(\rho-z)\|^2+(1-\alpha_j)\|y_{j}-z\|^{2}\\
                  &\leq (1-\alpha_j)\|x_j-z\|^{2}+\alpha_j \|(\rho-z)\|^2+\lambda_j\Gamma_{j} \\
                  & \quad -(1-\alpha_j)q_{j}\big(\|x_j-w_j\|^{2}+\|w_j-y_{j}\|^{2}\big)\\ 
                  &\leq (1-\alpha_j)\|x_j-z\|^{2}+\alpha_j \|(\rho-z)\|^2+\lambda_j\Gamma_{j}.
\end{split}
\end{equation}
Taking $H=2\underset{j\geq 1}{\sup}\Big\{\|(\rho-z)\|^2+\frac{\lambda_j}{\alpha_{j}}\Gamma_{j}\Big\}$ and the fact that $\alpha_{j}\in (0,1),$  \eqref{T7} becomes
\begin{equation}\label{TT3}
  \begin{split}
    \|x_{j+1}-z\|^{2}
    &\leq (1-\alpha_j)\|x_{j}-z\|^2 +\alpha_{j}H \\
    &\leq \max\{\|x_{j}-z\|^2,H\} \\
    &\quad\vdots\\
    &\leq \max\{\|x_{0}-z\|^2,H\}. \\
  \end{split}
\end{equation}
Similarly, combining \eqref{TT3} and the condition  $(A_{4})$  we can see that  $\{\|x_{j}-z\|\}$ is bounded with respect to $z\in VIP(B,C).$ Consequently, the sequences $\{x_{j}\},\; \{w_{j}\}$ and $\{y_{j}\}$ are also bounded.
\end{proof}

\begin{lemma} \label{T8}
  Let $\{w_{j}\},\; \{y_{j}\}\; \text{and}\; \{x_{j}\}$ be the sequences generated by Algorithm \ref{algo_W}. Then, for any $z\in VIP(B,C)$, we have
  \begin{equation}\label{strg2}
    \|x_{j+1}-z\|^{2}\leq ||x_{j}-z||^{2}-\zeta_{j}+\phi_{j}
  \end{equation}
  and
  \begin{equation}\label{strg1}
    ||x_{j+1}-z||^{2}\leq (1-\alpha_{j})||x_{j}-z||^{2}+\alpha_{j}V_{j}.
  \end{equation}
\end{lemma}
\begin{proof}
  Now, observe that
  \begin{equation}\label{T9}
    \begin{split}
      \left\langle \rho-z, x_{j+1}-z\right\rangle
        &= \left\langle \rho-z, \alpha_j \rho+(1-\alpha_{j})y_{j}-z\right\rangle\\
        &\leq  \alpha_j \left\langle \rho-z, \rho\right\rangle- \alpha_j\left\langle \rho-z,z\right\rangle
        +\left\langle \rho-z,(1-\alpha_{j})y_{j}-z\right\rangle\\
        & \leq \alpha_j \| \rho-z\|^{2}+(1- \alpha_j)\left\langle \rho-z,y_{j}-z\right\rangle.
    \end{split}
  \end{equation}
  By \eqref{T9}, we have
  \begin{equation}\label{p9}
    \begin{split}
      \|x_{j+!}-z\|^2
     &\leq (1-\alpha_j)\|x_j-z\|^{2}+(1-\alpha_j)\lambda_j\Gamma_{j}+2\alpha_{j} \left\langle \rho-z, x_{j+!}-z \right\rangle\\
     & \leq (1-\alpha_j)\|x_j-z\|^{2}+(1-\alpha_j)\lambda_j\Gamma_{j} +2\alpha_{j}^2   \| \rho-z\|^{2}\\
     &\quad+2\alpha_{j}(1- \alpha_j)\left\langle \rho-z,y_{j}-z\right\rangle.
\end{split}
\end{equation}
Taking 
$\Phi_{j}=\alpha_j \|(\rho-z)\|^2+\lambda_j\Gamma_{j},$\quad
$V_{j}=  (1-\alpha_j)\frac{\lambda_j}{\alpha_{j}}\Gamma_{j}  +2 (\alpha_j \| \rho-z\|^{2}+(1- \alpha_j)\left\langle \rho-z,y_{j}-z\right\rangle)$ and
$\xi_{j}=(1-\alpha_j)q_{j}\big(\|x_j-w_j\|^{2}+\|w_j-y_{j}\|^{2}\big),$ 
we easily see that \eqref{strg1} and \eqref{strg2} follow from \eqref{T7} and \eqref{p9} respectively. 
\end{proof}

Now, we formulate and prove the following as our main convergence theorem.
\begin{theorem}
  Suppose that the conditions of Assumption 1 hold and  $\{x_j\}$ is a sequence generated by Algorithm \ref{algo_W} 
  Then, the  sequence $\{x_j\}$  converges strongly to $z = P_{\text{VIP(B,C)}}(\rho)$.
\end{theorem}
\begin{proof}
  It is easy to see from the condition  $(A_{4}),$  and Lemma \eqref{T8} that $\underset{j\to\infty}{\lim}\Phi_j=0.$
  Thus, to apply Lemma \ref{T8} , we only need to show that for any subsequence $\{\xi_{j_{r}}\}$  of $\{\xi_{j}\},$ the following holds. 
  $$\underset{r\to\infty}{\lim}\; \xi_{j_{r}}=0\; \Rightarrow\; \underset{r\to\infty}{\limsup}\; 
  V_{j_{r}}\leq 0.$$
  Now, suppose  that  $\{\xi_{j_{r}}\}$ is a subsequence  of $\{\xi_{j}\}$  such that  $\underset{r\to\infty}{\lim}\; \xi_{j_{r}}=0,$  then,  in view of condition   $(A_{4}),$   one can see that
  \begin{equation}\label{TT1}
    \underset{r\to\infty}{\lim} \|x_{j_{r}}-w_{j_{r}}\|=0. 
  \end{equation}
  Since $\{x_{j}\}$ is bounded, then there exists a subsequence  $\{x_{j_{i}}\} $  of  $\{x_{j}\} $ 
  such that $x_{j_{i}} \rightharpoonup \tau$ as $i\rightarrow{\infty}.$ Combining  \eqref{TT02} and \eqref{TT1}, one sees that $y_{j_{i}} \rightharpoonup \tau.$ But $C$ is closed, therefore $\tau\in C.$ Now, we want to show that the weak cluster point  $\omega_{w}(x_{j})\subseteq VIP(B,C).$ By (i) of Definition \ref{def01}, we have
  $$\left\langle x_{j_{i}}+\lambda_{j_{i}}(-B(x_{j_{i}})+\Theta_{j_{i}}d_{j_{i}}))-w_{j_{i}}, w_{j_{i}}-s\right\rangle\geq 0, \quad \forall s\in C,$$
  it follows that
  $$\left\langle \frac{x_{j_{i}}-w_{j_{i}}}{\lambda_{j_{i}}}-B(x_{j_{i}})+\Theta_{j_{i}}d_{j_{i}}, w_{j_{i}}-s\right\rangle\geq 0, \quad \forall s\in C,$$
  which implies that
  $$\left\langle \frac{x_{j_{i}}-w_{j_{i}}}{\lambda_{j_{i}}}, w_{j_{i}}-s\right\rangle+\Theta_{j_{i}}\left\langle d_{j_{i}}, w_{j_{i}}-s\right\rangle+\left\langle B(x_{j_{i}}), s-w_{j_{i}}\right\rangle \geq  0, \quad \forall s\in C.$$
  It follows that
  \begin{align*}
  &\left\langle \frac{x_{j_{i}}-w_{j_{i}}}{\lambda_{j_{i}}}, w_{j_{i}}-s\right\rangle +\Theta_{j_{i}}\left\langle d_{j_{i}}, w_{j_{i}}-x_{j_i}\right\rangle \\
  &\qquad\qquad+\left\langle B(x_{j_{i}}), s-x_{j_{i}}\right\rangle +\left\langle B(x_{j_{i}}), x_{j_i}-w_{j_{i}}\right\rangle\geq  0, \quad \forall s\in C.
  \end{align*}
  $\text{As}\;  i\rightarrow{\infty},$ we have
  $$ \liminf\limits_{i\to\infty}\left\langle B(x_{j_{i}}), s-x_{j_{i}}\right\rangle\geq 0, \quad \forall s\in C.$$
  What justifies the existence of the sequence~$\{\rho_{i}\}$ ?
  Let take $\{\rho_{i}\}$ to be a positive decreasing sequence such that $\underset{i\to\infty}{\lim}\rho_i=0.$ Now, for each $\rho_i,$ let $\mu_{i}$ be the smallest nonnegative integer such that 
  \begin{equation}\label{cvb}
    \left\langle B(x_{j_{n}}), s-x_{j_{n}}\right\rangle+\rho_{i}\geq 0, \quad \forall n\geq \mu_i.
  \end{equation}
  Observe that $\{\mu_i\}$ is increasing from the fact that $\rho_{i}$ is decreasing. So now, for $B(x_{j_i})=0,$  we have $x_{j_i}\in V(B,C)$ for each $i,$  and for $B(x_{j_{\mu_i}})\neq 0,$ we let $\lambda_{j_{\mu_{i}}}=B(x_{j_{\mu_{i}}})\|B(x_{j_{\mu_{i}}})\|^{-2}.$ It follows that $\left\langle B(x_{j_{\mu_{i}}}), \lambda_{j_{\mu_{i}}}\right\rangle=1, \quad\text{for each} \,\,i.$ By \eqref{cvb}, we have
  \begin{equation}
    \left\langle  B(x_{j_{\mu_{i}}}), s+ \rho_{i} \lambda_{j_{\mu_{i}}}-x_{j_{\mu_{i}}}\right\rangle\geq 0. \quad  
  \end{equation}
  The pseudomonotonicity of $B,$ imply 
  \begin{equation}\label{cvb1}
    \left\langle  B(s+ \rho_{i} \lambda_{j_{\mu_{i}}}), s+ \rho_{i} \lambda_{j_{\mu_{i}}}-x_{j_{\mu_{i}}}\right\rangle\geq 0. \quad 
  \end{equation}
  The weakly converges of $\{x_{j_i}\}$ and the fact that $B$ is sequentially weakly continuous on $C,$  $B(x_{j_i})$ is weakly converges to $B(\tau).$ Now, suppose that 
  $B(\tau)\neq 0,$ by using the  fact that norm mapping is weakly  lower semicontinuous  we have $\|B(\tau)\|\leq \underset{i\to\infty}{\liminf\|B(x_{j_i})\|}.$ Observe also that $\{x_{j_{\mu_i}}\}\subset\{x_{j_i}\}$ and 
  \begin{equation}
    \underset{i\to \infty }{\lim}\frac{\rho_{i}}{\|B(x_{j_{\mu_i}})\|}\leq \frac{0}{\|B(\tau)\|}=0.
  \end{equation}
  As $i\rightarrow{\infty}$ in \eqref{cvb1}, we have 
  \begin{equation}\label{cvb10}
    \left\langle  B(s),s-\tau\right\rangle\geq 0. \quad 
  \end{equation}
  Thus, $\tau\in\omega_{w} (x_{j})\subseteq VIP(B,C).$  
  Now, to show that  the sequence $\{x_{j}\}$ converges weakly to a point in $VIP(B,C),$ It suffices to show that $\omega_{w} (x_{j})$ is singleton. From Lemma \ref{opial},
  taking $z_{1},z_{2}\in \omega_{w} (x_{j})$  and let  $\{x_{j_r}\}$ and $\{x_{j_k}\}$ be  subsequences of $\{x_j\}$ such that \(x_{j_r} \rightharpoonup z_{1}\) and  \(x_{j_k} \rightharpoonup z_{2}\). 
  Then, for $z_{1}\neq z_2,$ we obtain 
  \begin{align*}
    \lim_{j \to \infty} \|x_j - z_{1}\| &= \liminf_{r \to \infty} \|x_{j_r} - z_{1}\|< \liminf_{r \to \infty} \|x_{j_r} - z_{2}\| \\
                                        &= \lim_{j \to \infty} \|x_j - z_{2}\|= \liminf_{k \to \infty} \|x_{j_k} - z_{2}\|< \liminf_{k \to \infty} \|x_{j_k} - z_{1}\| \\
                                        &= \lim_{j \to \infty} \|x_j - z_{1}\|,
  \end{align*}
  which is a contradiction. Hence $ \omega_{w} (x_{j})$ is singleton. 

  By the fact that  $z\in VIP(B,C),$ and metric projection property given
  in (i) of definition \ref{def01}, we have
  \begin{equation}\label{Tq9}
    \begin{split}
      \limsup_{j\rightarrow{\infty}}\left\langle \rho-z, y_{j}-z\right\rangle=&\lim_{i\rightarrow{\infty}}\left\langle \rho-z,  y_{j_{i}}-z\right\rangle\\
      = & \left\langle \rho-z, z-z\right\rangle\\&\leq 0.
    \end{split}
  \end{equation}
  Together with  condition  $(A_{4}),$ we have  $\limsup_{i\rightarrow{\infty}}V_{j_{i}}\leq 0.$ It finally follows from Lemma \ref{con} that 
  \begin{equation}\label{ttT9}
    \lim_{n\rightarrow{\infty}} \|x_{j}-z\|=0,
  \end{equation}
  thus, $x_{j}\rightarrow z$ as $j\rightarrow{\infty}.$
\end{proof}

\section{Numerical Example}\label{zec3}

In this section, our goal is to demonstrate the efficiency and robustness of the proposed strategy, denoted by (Algo3), in comparison with the Double Inertial Steps Into the Single Projection Method with Non-Monotonic Step Sizes for Solving Pseudomonotone Variational Inequalities \cite{thong2024using}, the Two Subgradient Extragradient Methods Based on the Golden Ratio Technique for Solving Variational Inequality Problems \cite{oyewole2024two}, and the Versions of the Subgradient Extragradient Method for Pseudomonotone Variational Inequalities \cite{khanh2020versions}, respectively denoted by (Algo1), (Algo2), and (Algo4). We used Matlab 2021a on a Dell Core i7 computer to conduct the numerical simulation. 

\begin{example}\label{EEe1}
  We consider the nonlinear VIP proposed by Sun in \cite[Example 5]{sun1994projection}, which reads:
  \[ \text{Find}\quad x^{*}\in C \quad\text{such that}\quad
    \langle \Phi(x^*), x - x^* \rangle \geq 0, \quad \forall x \in C,
  \]
  where
  \[
    \Phi(x) = B_1(x) + B_2(x),
  \]
  \[
    B_1(x) = (b_1(x), b_2(x), \dots, b_d(x)),
  \]
  \[
    B_2(x) = Dx + c,
  \]
  \[
    b_i(x) = x_i^2 + x_i + x_{i-1}x_i + x_ix_{i+1}, \quad i = 1, 2, \dots, d,
  \]
  with boundary conditions \( x_0 = x_{d+1} = 0 \).
  The matrix \( D \) is a square matrix of size \( d \times d \), defined by:
  \[
    d_{ij} = 
    \begin{cases} 
      4, & i = j, \\
      1, & i - j = 1, \\
      -2, & i - j = -1, \\
      0, & \text{otherwise}.
    \end{cases}
  \]
  where \( c = (-1, -1, \dots, -1) \in \mathbb{R}^d \) and the  the feasible set \( C = \mathbb{R}_+^d \).
  We set  the control parameters  for example 1 as follows:
  \begin{enumerate}[label=Algo\arabic*:, leftmargin=*]
    \item $\mu=0.1,\, \lambda_{0}=0.5, \beta=0.8$\, and\, $\delta_{n}=\dfrac{1}{(n^2+1)^2} $
    \item $\mu=0.5, t=1.4, \, \lambda_{0}=0.03, \delta_{n}=\dfrac{2}{(7n+1)}\,$ and \,$\beta_n=1 +\dfrac{1}{n^{0.33}} $
    \item $\mu=0.05,\, \lambda_{0}=0.03, \,\, \rho=0.1,\alpha_{n}=\dfrac{1}{(2n+1)^{0.01}} \, \text{and}\, \,\Gamma_{n}=\dfrac{1}{(n+1)^{1.2}}$
    \item $\mu=0.1 \, \lambda_{0}=0.05,
      \,\beta^{1}_{n}=\dfrac{1}{(n+1)^{0.001}}+1 \, \,\text{and} \quad \beta^{2}_{n}=0.4(1-\beta^{1}_{n}) $.
  \end{enumerate}  
  The initial point \( x_0 \) is uniformly generated and
  we conducted a numerical experiment to examine the performance of our methods and the other three algorithms. The performance are compared in the Figure~\ref{fig1m} below.
  \begin{figure}[H]
    \centering
    \includegraphics[width=0.8\textwidth]{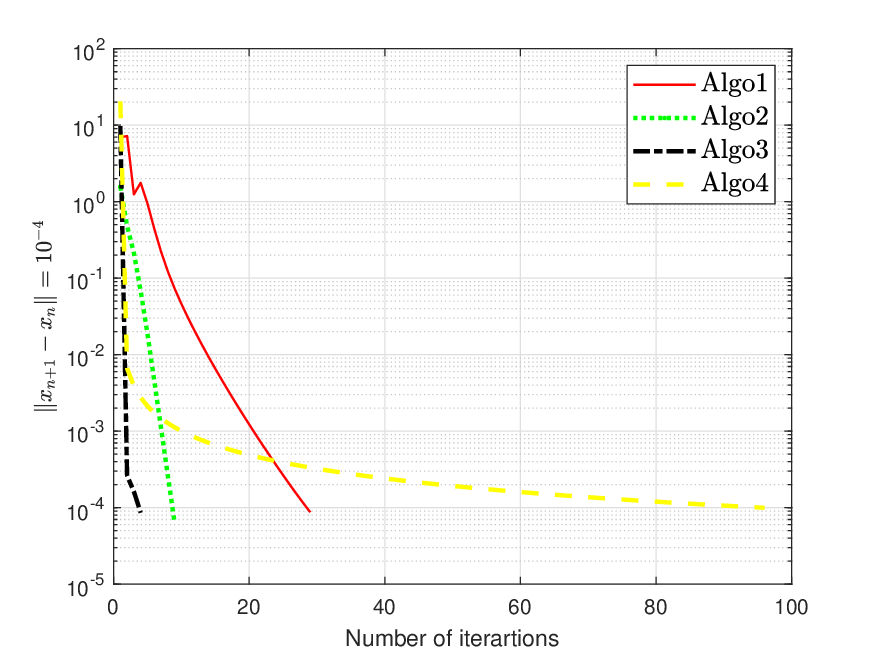}
    \caption{Convergence plots for Example \ref{EEe1} with dimension d=100}
    \label{fig1m}
  \end{figure}\noindent \hspace*{-0.1 cm}
  It is evident from Figure \ref{fig1m}  that the proposed method is effective and requires fewer number of iterations to converge, followed by Algo2. However, Algo1 and Algo4 converge with a greater number of iterations. The results demonstrate the computational efficiency and robustness of the proposed approach. \qed
\end{example}

\begin{example}\label{EE22}
  Consider ${H} = \ell^2$, the real Hilbert space whose elements are square-summable infinite sequences of real numbers, and let $C = \{ x_n \in H: \| x_n\| \leq 5 \}$. The operator $\Phi$ is given by
  \[
    \Phi(x_n) = (7 - \| x_n \|)x_n,\quad \forall x_n \in {H},
  \]
  where $\| x_n \| = \sqrt{\sum_i |x_i|^2}$. It is easy to see that the variational inequality problem (VIP) for $\Phi$ on $C$, denoted by $\text{VIP}(\Phi, C)$, is non-empty. The operator $A$ is pseudomonotone with a Lipschitz constant $L = 11$. The projection onto $C$ is explicitly given by:
  \[P_C(u) = 
    \begin{cases} 
      x_n, & \text{if }\, \| x_n \| \leq 5, \\
      \frac{5x_n}{\| x_n \|}, & \text{otherwise}.
    \end{cases}
  \]
  We set $n=50$ and the control parameters as follows:
  \begin{enumerate}[label=Algo\arabic*:, leftmargin=*]
    \item $\mu=0.05,\, \lambda_{0}=0.03 \,\,\text{and} \,\,\delta_{n}=\frac{1}{(n^2+1)^{0.01}} .$
    \item $\mu=0.5, t=1.4, \, \lambda_{0}=0.03 \,\,\text{and} \,\,\delta_{n}=\frac{2}{(7n+1)} .$
    \item all parameters remain the same except  for $\rho= 0.2$\\ 
    \item $\mu=0.5, \, \lambda_{0}=0.03,
      \,\beta^{1}_{n}=\dfrac{2}{(7n+1)} \,\,\text{and} \,\, \beta^{2}_{n}=0.7(1-\beta^{1}_{n}). $
  \end{enumerate}
  We examined the performance of the proposed with the other three methods, the result is given in the Figure \ref{fig2e} below.
  \begin{figure}[H]
    \centering
    \includegraphics[width=0.8\textwidth]{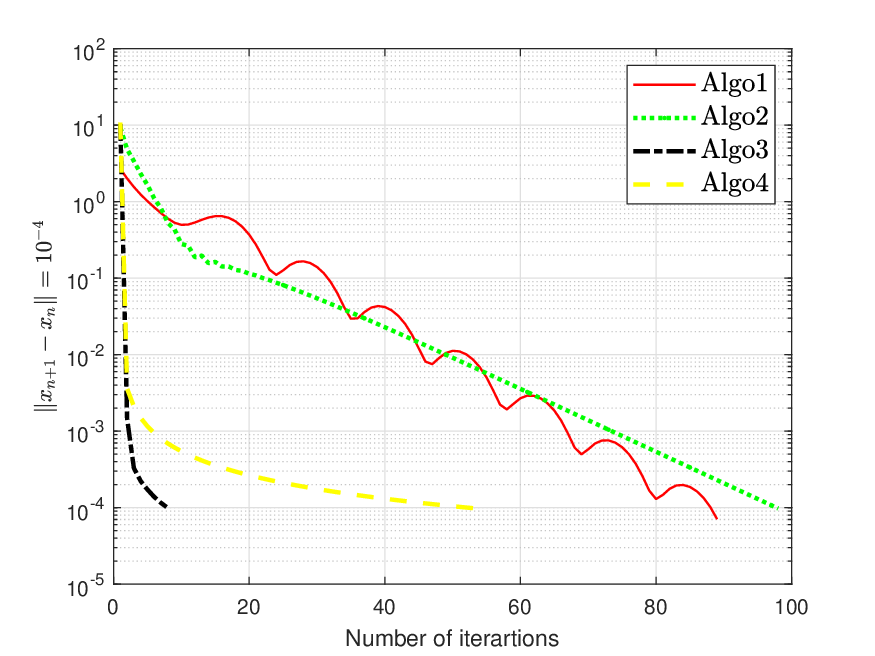}
    \caption{Convergence plots for \ref{EE22} with dimension n=100}
    \label{fig2e}
  \end{figure}
  In this example, we set the tolerance to $10^{-4}$ and present the numerical result in the Figure \ref{fig2e}. The graphical representation clearly shows that the proposed method outperforms all the three methods considered, followed by Algo4, and then the remaining method. \qed
\end{example}

Next, we consider the application of the proposed method in the matrix game problem
\begin{example}[\cite{nemirovski2013mini}] \label{EE33}
  Let us consider a matrix game problem between policeman and the Burglars. There are \( n \) houses in a city, where the \( i \)-th house has wealth \( w^*_i \). Every evening, a Burglar chooses a house \( i \) to attack, and a Policeman chooses to post himself near a house \( j \). After the burglary starts, the Policeman becomes aware of where it happens, and his probability to catch the Burglar is:
  \[
    \exp(-\alpha \cdot \text{dist}(i, j)),
  \]
  where \( \text{dist}(i, j) \) is the distance between houses \( i \) and \( j \), and \( \alpha \) is a constant. The Burglar seeks to maximize his expected profit
  \[
    w^*_i(1 - \exp(-\alpha \cdot \text{dist}(i, j))),
  \]
  while the interest of the Policeman is the opposite. The matrix game optimization problem is given by 
  \[
    \min_{\substack{u }} \max_{\substack{v}} \phi(u, v) := v^T \Phi u,
  \]
  where
  \[
    \Phi_{ij} = w^*_i \left( 1 - \exp(-\alpha \cdot \text{dist}(i, j)) \right)
  \]    
  solves the matrix game 
The control parameters are presented by:
\begin{enumerate}[label=Algo\arabic*:, leftmargin=*]
  \item $\mu=0.1,\, \lambda_{0}=0.03 \,\text{and}\; \;\delta_{n}=\frac{1}{(n^2+1)^{0.01}} $
  \item $\mu=0.5, t=1.4, \, \lambda_{0}=0.03, \, \text{and}\,\;\delta_{n}=\frac{2}{(7n+1)} $
  \item $\mu= 0.07\,\, \text{and}\,\, \rho=0.2$
  \item $\mu=0,5,\, \,\text{and}\,\,\lambda_{0}=0.03,
    \,\beta^{1}_{n}=\frac{2}{(7n+1)} \,\, \text{and} \quad \beta^{2}_{n}=0.7(1-\beta^{1}_{n}). $
\end{enumerate}
\begin{figure}[H]
  \centering
  \includegraphics[width=0.8\textwidth]{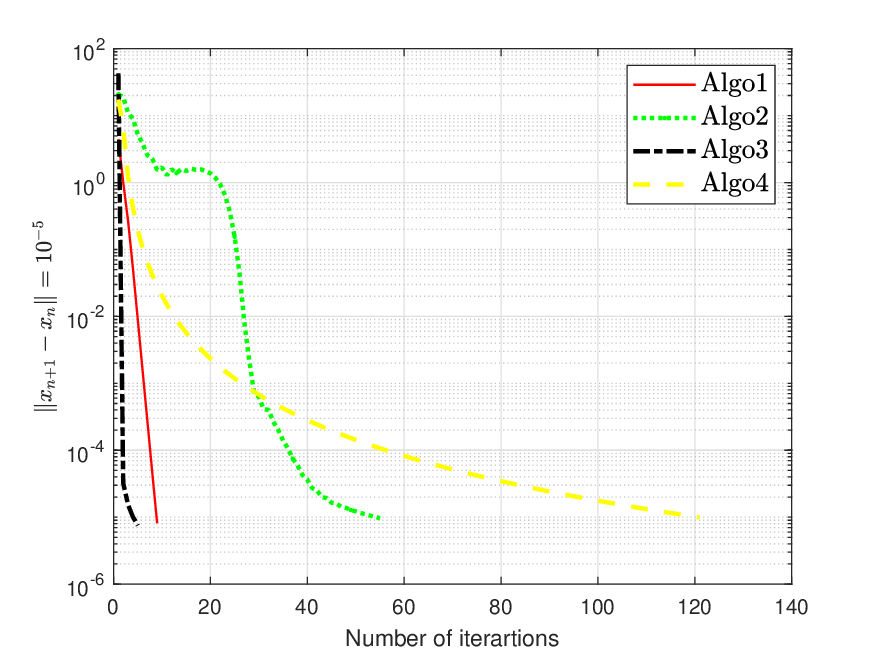}
  \caption{Convergence plots for Example \ref{EE33}  with $w_{j}=50$}
  \label{fig1aa}
\end{figure}
In this example, we set the number of wealthy houses $w_{j}$ to be  50 and $10^{-5}$ to be the stopping criterion. As depicted in Figure \ref{fig1aa}, the the performance of the proposed method is highly encouraging, followed by Algo1, Alglo2 and then Algo3 respectively. \qed
\end{example}

\section{Conclusion}\label{zec4}
The paper presented an adaptive-type subgradient extragradient algorithm that incorporates a conjugate gradient–type direction. The adaptive nature is associated with the steplength, whose sequence is shown to be nonincreasing and bounded away from zero. The selection of the steplength aids in establishing the boundedness of the main iterates generated by the algorithm. Consequently, we establish an inequality that describes the geometric nature of the main iterate relative to the solution elements.  Utilizing these inequalities and the Halpern term in the algorithm, guaranteed the strong convergence of the algorithm, which is the main convergence result. To show the effectiveness of the proposed algorithm, we provided numerical illustrations and also presented an application in which the performance of the algorithm was investigated empirically. The results were obtained using MATLAB and are displayed as figures.

\small
\section*{Acknowledgments}
The first author was supported by Petchra Pra Jom Klao Ph.D. Research Scholarship from King Mongkut’s University of Technology Thonburi, Thailand (Grant No.55/2021).
The authors acknowledge the financial support provided by the Center of Excellence in Theoretical and Computational Science (TaCS-CoE), KMUTT.

\section*{Conflicts of interests}
The authors state no conflicts of interests.

\section*{Author contributions}
\textbf{Conceptualization}, Poom Kumam;
\textbf{Supervision}, Parin Chaipunya;
\textbf{Investigation}, Ibrahim Arzuka;
\textbf{Writing - original draft}, Ibrahim Arzuka and Parin Chaipunya; 
\textbf{Review \& editing}, Poom Kumam.

\renewcommand\bibname{References}


\begin{thebibliography}{30}
\providecommand{\natexlab}[1]{#1}
\providecommand{\url}[1]{\texttt{#1}}
\expandafter\ifx\csname urlstyle\endcsname\relax
  \providecommand{\doi}[1]{doi: #1}\else
  \providecommand{\doi}{doi: \begingroup \urlstyle{rm}\Url}\fi

\bibitem[Abubakar et~al.(2022)Abubakar, Kumam, and Rehman]{abubakar2022self}
J.~Abubakar, P.~Kumam, and H.~U. Rehman.
\newblock Self-adaptive inertial subgradient extragradient scheme for pseudomonotone variational inequality problem.
\newblock \emph{International Journal of Nonlinear Sciences and Numerical Simulation}, 23\penalty0 (1):\penalty0 77--96, 2022.

\bibitem[Adamu et~al.(2022)Adamu, Kitkuan, and Seangwattana]{adamu2022accel}
A.~Adamu, D.~Kitkuan, and T.~Seangwattana.
\newblock An accelerated halpern-type algorithm for solving variational inclusion problems with applications.
\newblock \emph{Bangmod International Journal of Mathematical and Computational Science}, 8:\penalty0 37--55, 2022.

\bibitem[Alakoya et~al.(2021)Alakoya, Jolaoso, and Mewomo]{alakoya2021modified}
T.~Alakoya, L.~Jolaoso, and O.~Mewomo.
\newblock Modified inertial subgradient extragradient method with self adaptive stepsize for solving monotone variational inequality and fixed point problems.
\newblock \emph{Optimization}, 70\penalty0 (3):\penalty0 545--574, 2021.

\bibitem[Censor et~al.(2011)Censor, Gibali, and Reich]{censor2011strong}
Y.~Censor, A.~Gibali, and S.~Reich.
\newblock Strong convergence of subgradient extragradient methods for the variational inequality problem in hilbert space.
\newblock \emph{Optimization Methods and Software}, 26\penalty0 (4-5):\penalty0 827--845, 2011.

\bibitem[Censor et~al.(2012)Censor, Gibali, and Reich]{censor2012}
Y.~Censor, A.~Gibali, and S.~Reich.
\newblock Extensions of korpelevich's extragradient method for the variational inequality problem in euclidean space.
\newblock \emph{Optimization}, 61\penalty0 (9):\penalty0 1119--1132, 2012.

\bibitem[He and Yang(2013)]{he2013solving}
S.~He and C.~Yang.
\newblock Solving the variational inequality problem defined on intersection of finite level sets.
\newblock In \emph{Abstract and applied analysis}, volume 2013, page 942315. Wiley Online Library, 2013.

\bibitem[Iiduka(2011)]{iiduka2011three}
H.~Iiduka.
\newblock Three-term conjugate gradient method for the convex optimization problem over the fixed point set of a nonexpansive mapping.
\newblock \emph{Applied mathematics and computation}, 217\penalty0 (13):\penalty0 6315--6327, 2011.

\bibitem[Iiduka(2012{\natexlab{a}})]{iiduka2012afixed}
H.~Iiduka.
\newblock Fixed point optimization algorithm and its application to power control in cdma data networks.
\newblock \emph{Mathematical Programming}, 133:\penalty0 227--242, 2012{\natexlab{a}}.

\bibitem[Iiduka(2012{\natexlab{b}})]{iiduka2012fixeda}
H.~Iiduka.
\newblock Fixed point optimization algorithm and its application to network bandwidth allocation.
\newblock \emph{Journal of Computational and Applied Mathematics}, 236\penalty0 (7):\penalty0 1733--1742, 2012{\natexlab{b}}.

\bibitem[Iiduka(2015)]{iiduka2015acce}
H.~Iiduka.
\newblock Acceleration method for convex optimization over the fixed point set of a nonexpansive mapping.
\newblock \emph{Mathematical Programming}, 149\penalty0 (1):\penalty0 131--165, 2015.

\bibitem[Iiduka and Yamada(2009)]{iiduka2009use}
H.~Iiduka and I.~Yamada.
\newblock A use of conjugate gradient direction for the convex optimization problem over the fixed point set of a nonexpansive mapping.
\newblock \emph{SIAM Journal on Optimization}, 19\penalty0 (4):\penalty0 1881--1893, 2009.

\bibitem[Khanh et~al.(2020)Khanh, Thong, and Vinh]{khanh2020versions}
P.~Q. Khanh, D.~V. Thong, and N.~T. Vinh.
\newblock Versions of the subgradient extragradient method for pseudomonotone variational inequalities.
\newblock \emph{Acta Applicandae Mathematicae}, 170\penalty0 (1):\penalty0 319--345, 2020.

\bibitem[Konnov(1997)]{konnov1997systems}
I.~V. Konnov.
\newblock On systems of variational inequalities.
\newblock \emph{Russian Mathematics c/c of Izvestiia-Vysshie Uchebnye Zavedeniia Matematika}, 41:\penalty0 77--86, 1997.

\bibitem[Korpelevich(1976)]{korpelevich19}
G.~M. Korpelevich.
\newblock The extragradient method for finding saddle points and other problems.
\newblock \emph{Matecon}, 12:\penalty0 747--756, 1976.

\bibitem[Kravchuk and Neittaanm{\"a}ki(2007)]{kravchuk2007variational}
A.~S. Kravchuk and P.~J. Neittaanm{\"a}ki.
\newblock \emph{Variational and quasi-variational inequalities in mechanics}, volume 147.
\newblock Springer Science \& Business Media, 2007.

\bibitem[Nagurney and Nagurney(2012)]{2012medical}
A.~Nagurney and L.~S. Nagurney.
\newblock Medical nuclear supply chain design: A tractable network model and computational approach.
\newblock \emph{International journal of production economics}, 140\penalty0 (2):\penalty0 865--874, 2012.

\bibitem[Nemirovski(2013)]{nemirovski2013mini}
A.~Nemirovski.
\newblock Mini-course on convex programming algorithms.
\newblock \emph{Lecture notes}, 2013.

\bibitem[Opial(1967)]{opial1967weak}
Z.~Opial.
\newblock Weak convergence of the sequence of successive approximations for nonexpansive mappings.
\newblock \emph{Bulletin of the American Mathematical Society}, 73\penalty0 (4):\penalty0 591--597, 1967.

\bibitem[Oyewole and Reich(2024)]{oyewole2024two}
O.~K. Oyewole and S.~Reich.
\newblock Two subgradient extragradient methods based on the golden ratio technique for solving variational inequality problems.
\newblock \emph{Numerical Algorithms}, pages 1--22, 2024.

\bibitem[Phairatchatniyom et~al.(2021)Phairatchatniyom, ur~Rehman, Abubakar, Kumam, Mart, et~al.]{yom2021inertial}
P.~Phairatchatniyom, H.~ur~Rehman, J.~Abubakar, P.~Kumam, J.~Mart, et~al.
\newblock An inertial iterative scheme for solving split variational inclusion problems in real hilbert spaces.
\newblock \emph{Bangmod International Journal of Mathematical and Computational Science}, 7:\penalty0 35--52, 2021.

\bibitem[Popov(1980)]{popov1980mod}
L.~D. Popov.
\newblock A modification of the arrow-hurwitz method of search for saddle points.
\newblock \emph{Mat. Zametki}, 28\penalty0 (5):\penalty0 777--784, 1980.

\bibitem[Saberi(2016)]{saberi2016network}
S.~Saberi.
\newblock \emph{Network Game Theory Models of Services and Quality Competition with Applications to Future Internet Architectures and Supply Chains}.
\newblock University of Massachusetts Amherst, 2016.

\bibitem[Semenov and Semenova(2005)]{semenov2005vector}
V.~Semenov and N.~Semenova.
\newblock A vector problem of optimal control in a hilbert space.
\newblock \emph{Cybernetics and Systems Analysis}, 41:\penalty0 255--266, 2005.

\bibitem[Solodov and Svaiter(1999)]{solodov1999}
M.~V. Solodov and B.~F. Svaiter.
\newblock A new projection method for variational inequality problems.
\newblock \emph{SIAM Journal on Control and Optimization}, 37\penalty0 (3):\penalty0 765--776, 1999.

\bibitem[Sun(1994)]{sun1994projection}
D.~Sun.
\newblock A projection and contraction method for the nonlinear complementarity problems and its extensions.
\newblock \emph{Math Numer Sinica}, 16:\penalty0 183--194, 1994.

\bibitem[Thong et~al.(2024)Thong, Li, Dung, Huyen, and Tam]{thong2024using}
D.~V. Thong, X.-H. Li, V.~T. Dung, P.~T.~H. Huyen, and H.~T.~T. Tam.
\newblock Using double inertial steps into the single projection method with non-monotonic step sizes for solving pseudomontone variational inequalities.
\newblock \emph{Networks and Spatial Economics}, 24\penalty0 (1):\penalty0 1--26, 2024.

\bibitem[Tian and Tong(2019)]{tian2019self}
M.~Tian and M.~Tong.
\newblock Self-adaptive subgradient extragradient method with inertial modification for solving monotone variational inequality problems and quasi-nonexpansive fixed point problems.
\newblock \emph{Journal of Inequalities and Applications}, 2019:\penalty0 1--19, 2019.

\bibitem[ur~Rehman et~al.(2022)ur~Rehman, Kumam, and Sombut]{ur2022inertial}
H.~ur~Rehman, W.~Kumam, and K.~Sombut.
\newblock Inertial modification using self-adaptive subgradient extragradient techniques for equilibrium programming applied to variational inequalities and fixed-point problems.
\newblock \emph{Mathematics}, 10\penalty0 (10):\penalty0 1751, 2022.

\bibitem[Wang et~al.(2022)Wang, Tao, Lin, and Cho]{wang2022self}
S.~Wang, H.~Tao, R.~Lin, and Y.~J. Cho.
\newblock A self-adaptive stochastic subgradient extragradient algorithm for the stochastic pseudomonotone variational inequality problem with application.
\newblock \emph{Zeitschrift f{\"u}r angewandte Mathematik und Physik}, 73\penalty0 (4):\penalty0 164, 2022.

\bibitem[Yang et~al.(2018)Yang, Liu, and Liu]{yang2018modified}
J.~Yang, H.~Liu, and Z.~Liu.
\newblock Modified subgradient extragradient algorithms for solving monotone variational inequalities.
\newblock \emph{Optimization}, 67:\penalty0 2247--2258, 2018.

\end{thebibliography}

\end{document}